\documentclass[11pt,a4paper]{amsart}
\usepackage{amssymb,amsmath,epsfig,graphics,mathrsfs,enumerate,verbatim}
\usepackage[pagebackref,colorlinks=true,linkcolor=blue,citecolor=blue]{hyperref}
\usepackage{fancyhdr}
\usepackage{hyperref}
\hypersetup{
 colorlinks   = true,
 urlcolor     = blue,
 linkcolor    = blue,
 citecolor   = red ,
 bookmarksopen=true
}

\usepackage{amsmath}
\usepackage{amsfonts}
\usepackage{amssymb}
\usepackage{amsthm}
\usepackage{epsfig,graphics,mathrsfs}
\usepackage{graphicx}
\usepackage{dsfont}

\usepackage[usenames, dvipsnames]{color}

\usepackage{hyperref}

\def \phi {\varphi}

\def \R {\mathbb{R}}

\def \vf{\varphi}

\def \So {\mathscr{S}(\Rm)}

\newcommand{\Rn}{\mathbb R^n}
\newcommand{\Rm}{\mathbb R^m}

\newcommand{\p}{\partial}

\newcommand{\la}{\lambda}

\numberwithin{equation}{section}

\newcommand{\beq}{\begin{equation}}
\newcommand{\bea}[1]{\begin{array}{#1} }
\newcommand{\eeq}{ \end{equation}}
\newcommand{\ea}{ \end{array}}

\newcommand{\I}{\mathscr I_{HL}}

\newcommand{\sa}{\langle}
\newcommand{\da}{\rangle}

\newtheorem{theorem}{Theorem}[section]
\newtheorem{lemma}[theorem]{Lemma}
\newtheorem{proposition}[theorem]{Proposition}
\newtheorem{corollary}[theorem]{Corollary}
\newtheorem{remark}[theorem]{Remark}

\numberwithin{equation}{section}

\begin{document}

\title[Dispersive equations with invariant measures]{Dispersive equations with invariant measures}

\dedicatory{I dedicate this paper to Carlos Kenig,  deeply influential mathematician, on the occasion of his birthday}



\date{}

\begin{abstract}
In mathematical physics it is of interest to study Schr\"odinger equations with friction and possessing an invariant measure. The focus of this paper is the Cauchy problem for the Schr\"odinger equation $\p_t f - i \mathscr L f = 0$, where  $\mathscr L = \Delta - \sa x,\nabla\da$ is the Ornstein-Uhlenbeck operator. We use this as a model to stimulate interest in a new class of possibly degenerate dispersive equations which cannot be treated by the existing theory.   
\end{abstract}

\author{Nicola Garofalo}

\address{School of Mathematical and Statistical Sciences\\ Arizona State University\\ Tempe, AZ 85287-1804}
\vskip 0.2in
\email{Nicola.Garofalo@asu.edu}

\thanks{The author was supported in part by a Progetto SID (Investimento Strategico di Dipartimento): ``Aspects of nonlocal operators via fine properties of heat kernels", University of Padova (2022); and by a PRIN (Progetto di Ricerca di Rilevante Interesse Nazionale) (2022): ``Variational and analytical aspects of geometric PDEs". He was also partially supported by a Visiting Professorship at the Arizona State University}

\maketitle

%\tableofcontents

\section{Introduction}\label{S:intro}

The Cauchy problem for the free Schr\"odinger equation
\begin{equation}\label{S}
\p_t f - i \Delta f = 0, \ \ \ \ \ \ \ f(x,0) = \vf(x),
\end{equation}
occupies a central position in quantum mechanics, and it represents a fundamental model for a class of equations known as dispersive. The name comes from the following estimate fulfilled by a solution $f(x,t)$ of \eqref{S}: for any $1\le p\le 2$, one has for some universal $C(m,p)>0$,  
\begin{equation}\label{class}
||f(\cdot,t)||_{L^{p'}(\Rm)} \le \frac{C(m,p)\ }{t^{m(\frac 12 - \frac 1{p'})}} ||\vf||_{L^p(\Rm)}.
\end{equation}
As in the theorem of Hausdorff-Young for the Fourier transform, \eqref{class} fails for $p>2$, see e.g. \cite[Lemma 1.2]{GV} or \cite[Proposition 2.2.3]{Caze}. One important application of \eqref{class} is the $T^\star T$ approach of Ginibre \& Velo to prove mixed norms inequalities, known as Strichartz estimates, for the nonlinear Schr\"odinger equation, see \cite{GVstrich} or \cite[Section 2.3]{Caze}.

In mathematical physics it is of interest to study Schr\"odinger equations with friction, but with a principal part that can be possibly degenerate. For instance, the Cauchy problem in $\R^2\times(0,\infty)$
\begin{equation}\label{sk}
 \p_t f = i\left(\p_{xx} f  - 2(x+y) \p_x f + x \p_y f\right),\ \ \ \ \ f(x,y,0) = \vf(x,y),
\end{equation}
arises from the Smoluchowski-Kramers' approximation of Brownian motion with drift, see \cite{Bri} and \cite{Fre}. One would like to know whether for \eqref{sk} there exist dispersive estimates such as \eqref{class}, local or global. However, since the diffusive term $\p_{yy}u$ is missing in \eqref{sk}, such equation cannot be treated by ad hoc modifications of the existing theory for \eqref{S}, and one has to develop new tools.
Equations such as \eqref{sk} fall into a class of open problems whose investigation is the background motivation of the present work. 

To describe the questions we have in mind, given $m\in \mathbb N$, consider  the second order partial differential operator in $\Rm$ defined by 
\begin{equation}\label{L}
\mathscr L f = \operatorname{tr}(Q \nabla^2 f) +  \langle B x,\nabla f\rangle.
\end{equation}
In \eqref{L} we have denoted by $Q$ and $B$ two $m\times m$ matrices with real constant coefficients. We assume that $Q = Q^\star \ge 0$, but this matrix can a priori be degenerate. With $\mathscr L$ as in \eqref{L}, in the opening of his fundamental work \cite{Ho}, H\"ormander considered the  Cauchy problem  in $\Rm\times (0,\infty)$
\begin{equation}\label{cpH}
\p_t f -  \mathscr L f = 0,\ \ \ \ \ \ \ \ \ f(x,0) = \vf(x),
\end{equation}
and proved that the PDE in \eqref{cpH} is hypoelliptic if and only if, for at least one $t>0$, the following condition holds for the covariance matrix generated by $Q$ and $B$,
\begin{equation}\label{Ds}
Q(t) = \int_0^t e^{sB}Q e^{sB^\star} ds > 0.
\end{equation}
He did so by showing that, under the assumption \eqref{Ds}, the problem \eqref{cpH} admits the ``heat kernel" 
\begin{equation}\label{q}
G(x,y,t) = \frac{(4\pi)^{-m/2}}{\sqrt{\det Q(t)}} \exp\left( -
\frac{\sqrt{\langle Q(t)^{-1}(y-e^{tB} x),y-e^{tB}
x \rangle}}{4}\right),
\end{equation}
which is easily checked to be $C^\infty$ outside the diagonal. It is worth noting here that the condition \eqref{Ds} identified by H\"ormander is also necessary and sufficient for the controllability of a first-order system, see \cite{Z}.

The present work is motivated by the following broad question: to study the Cauchy problem 
\begin{equation}\label{cpHi}
\p_t f -  i \mathscr L f = 0,\ \ \ \ \ \ \ \ \ f(x,0) = \vf(x),
\end{equation}
when the H\"ormander's condition \eqref{Ds} is satisfied. This is a problem which, at the moment, is unchartered territory: here, we present some initial progress by focusing the attention on the specific model \eqref{cp0} below. 

To put our results in the proper perspective, we note that the task of understanding \eqref{cpHi} is somewhat ambiguous since the class \eqref{L} encompasses equations with a very diverse physical background. To explain this point we note that the model problem \eqref{sk} presents two distinctive features:
\begin{itemize}
\item[(i)] an underlying complex Lie group invariance; 
\item[(ii)] the existence of an invariant measure.
\end{itemize}
While (i) is shared by the general Schr\"odinger equation \eqref{cpHi}, property (ii) depends on the drift matrix $B$ and it is not always true. 
Concerning (i), we mention that if one defines the left-multiplication operator
\[
\sigma_{(x,s)} (y,t) = (x,s) \circ (y,t) = (y + e^{-tB}x,t+s),
\]
then it is not difficult to verify the following commutation property for the PDE in \eqref{cpH}   
\[
(\p_t  -  \mathscr L)(\sigma_{(x,s)} f) = \sigma_{(x,s)}[(\p_t  -  \mathscr L) f].
\]
This invariance, which is reflected in the expression \eqref{q}, was first observed in the paper \cite{LP}. The problem \eqref{cpHi} displays a related, yet very different property. If in fact one considers the complex non-commutative group law 
\begin{equation}\label{complexLie}
\tau_{(x,s)} (y,t) = (x,s)\circ (y,t) = (y + e^{-itB}x,t+s),
\end{equation}
then it is not difficult to verify that 
\[
(\p_t  - i \mathscr L)(\tau_{(x,s)} f) = \tau_{(x,s)}[(\p_t  - i \mathscr L) f].
\]
The Lie group action \eqref{complexLie} links the analysis of \eqref{cpHi} to some interesting unexplored aspects of conformal geometry. In such broader perspective,  a motivation for \eqref{cpHi} is connected to the study of the Schr\"odinger equation $\p_t u - i \mathscr P u = 0$, where $\mathscr P$ is an ``hybrid" operator as defined in \cite{GT}. Yet another motivation is provided by the study of the Schr\"odinger equation $\p_t u - i \Delta_H u = 0$, where $\Delta_H$ indicates a sub-Laplacian in a group of Heisenberg type. In this direction, important contributions concerning mixed restriction inequalities, or dispersive estimates, were given in the works  \cite{Mu}, \cite{BaPaXu}, \cite{Hi}, \cite{BaFeGa}, \cite{BaBaGa} and \cite{BaGa}. Strichartz estimates in complex semisimple Lie groups were obtained in \cite{Cha}, whereas various forms of uncertainty principles were established in the works \cite{SST}, \cite{Cha}, \cite{Vel} and \cite{LM}.

Concerning (ii), we recall  that a probability measure $\mu$ in $\Rm$ is called invariant with respect to $\mathscr L$ in \eqref{L} if for every $u\in C^2_b(\Rm)$ one has
\[
\int_{\Rm} \mathscr L u\ d\mu = 0.
\]
Now, it is known that an invariant measure exists if and only if all the eigenvalues of the drift matrix $B$ have strictly negative real part, i.e.,  
 the spectrum $\sigma(B)$ satisfies the assumption
\begin{equation}\label{invmeas}
\max\{\Re(\lambda)\mid \lambda\in \sigma(B)\}\ <\ 0.
\end{equation}  
It is also known that \eqref{invmeas} is necessary and sufficient for the following integral
\begin{equation}\label{measure}
Q_\infty \overset{def}{=}\int_0^\infty e^{sB} Q e^{s B^\star} ds
\end{equation}
to converge to a well-defined positive matrix. 
We observe here that, from the definition of the covariance matrix $Q(t)$ in \eqref{Ds}, it is easy to check the following monotonicity property  
\begin{equation}\label{mono}
Q(t+s) = Q(t) + e^{tB} Q(s) e^{tB^\star},\ \ \ \ \ \ \ s, t >0.
\end{equation}
Note that \eqref{mono} implies that \eqref{Ds} holds for one $t>0$ if and only if it does so for every $t>0$, and that furthermore the function $t\to Q(t)$ is strictly increasing in the sense of quadratic forms, so that it always makes sense to consider the formal limit $\underset{t\nearrow \infty}{\lim} Q(t) = Q_\infty$. For these results one should see the monographs \cite[Section 6]{DZ} and \cite{Bo}. 

For instance, for the Smoluchowski-Kramers' model \eqref{sk} one has
\begin{equation}\label{skeigenv}
Q = \begin{pmatrix} 1 & 0\\ 0& 0\end{pmatrix},\ \ \ \ \  B = \begin{pmatrix} -2 & -2\\ 1 & 0\end{pmatrix}.
\end{equation} 
Since the eigenvalues of $B$ are $\la = -1-i$ and $\bar \la = -1+i$, we infer that \eqref{invmeas} is verified, and we refer to \cite[Sec. 5]{FL} for the computation of the relevant invariant measure. Furthermore, H\"ormander's condition \eqref{Ds} holds, so that the equation  in \eqref{cpH} is hypoelliptic. One has in fact for any $t>0$
\[
\operatorname{det} Q(t) = \frac{e^{-2t}}{16} \left(\cosh(2t) + \cos(2t) - 2\right) > 0,
\]
see \cite[Sec. 3.1]{GTmathann}.  However, as we have mentioned above, invariant measures do not always exist. Consider, in fact, the equation of Kolmogorov from Brownian motion and the kinetic theory of gases
\begin{equation}\label{Kol}
\p_t u- \Delta_v u - \sa v,\nabla_x u\da = 0,
\end{equation}
first introduced in \cite{Kol} (see also the discussion in \cite{Strex} and the book \cite[Sec. 7.4]{Str}). In \eqref{Kol} we have indicated the spatial variables $v, x\in \Rn$, so that, with $m = 2n$, the PDE acts on functions defined in $\Rm\times (0,\infty)$. Despite the degenerate character of \eqref{Kol} (note the missing diffusive term $\Delta_x u$), the equation is hypoelliptic. This had already been proved by Kolmogorov himself in \cite{Kol} since he did write the following explicit fundamental solution (actually, he only discussed the case $n=1$, in which case $m=2$)
\[
G(x,y,\bar x,\bar y,t) = \frac{3^{\frac n2}}{(2\pi)^n} t^{-2n} \exp\left\{-\frac{1}{4t}\left(|x-\bar x|^2 + 12 |\frac{y-\bar y}t + \frac{x+\bar x}{2}|^2\right)\right\}.
\]
To see the hypoellipticity using H\"ormander's theorem, observe that one has for \eqref{Kol} 
\begin{equation}\label{QBkolmo}
Q = \begin{pmatrix} I_n & O_n\\ O_n & O_n\end{pmatrix},\ \ \ \ \ \  B = \begin{pmatrix} O_n & O_n\\I_n & O_n\end{pmatrix}.
\end{equation}
Since $B^k = O_m$ for every $k\ge 2$, a simple computation gives 
\begin{equation}\label{QK}
Q(t)= \int_0^t e^{sB}Q e^{sB^\star} ds = \begin{pmatrix} tI_n & \frac{t^2}2 I_n\\ \frac{t^2}2 I_n & \frac{t^3}3 I_n\end{pmatrix},
\end{equation}
and therefore the assumption \eqref{Ds} holds, since one has for $t>0$
\begin{equation}\label{detkol}
\det Q(t) = c(n) t^{4n} = c(m) t^{2m}>0.
\end{equation}
From  \eqref{QBkolmo} it is clear that the condition \eqref{invmeas} fails (and in fact \eqref{QK} shows that $Q_\infty$ is not defined), we thus conclude that \eqref{Kol} does not admit an invariant measure. 

This discussion should convince the reader that the Cauchy problem \eqref{cpHi} requires a different approach, depending on whether \eqref{invmeas} is valid or not.  We mention that, when \eqref{invmeas} is satisfied, and therefore $Q_\infty$ in \eqref{measure} is well-defined, the invariant measure for $\mathscr L$ is given by
\begin{equation}\label{dgamma}
d\gamma_\infty(x) = \frac{(4\pi)^{-\frac m2}}{(\det Q_\infty)^{1/2}} e^{-\frac{\sa Q_\infty^{-1}x,x\da}4} dx,
\end{equation}
see e.g. \cite[Proposition 2.3]{BG&liar}. Since in such case the operator $\mathscr L$ is self-adjoint in $L^2(\Rm,d\gamma_\infty)$, from Stone's theorem (see e.g. \cite[Theorem 1, p. 345]{Yo}) we conclude that, associated with \eqref{cpHi}, there exists a strongly-continuous group $e^{i t \mathscr L}$ of unitary operators in $L^2(\Rm,d\gamma_\infty)$. As a first objective, one would like to obtain for such group a dispersive estimate similar to the classical one in \eqref{class}.

 This paper provides an initial contribution to this program by considering the model case $Q = I_m$ and $B = - I_m$ in \eqref{L}, so that 
\[
\mathscr L f= \Delta f - \sa x,\nabla f\da,
\] 
and \eqref{cpHi} presently becomes  
\begin{equation}\label{cp0}
\p_t f - i \left(\Delta f - \sa x,\nabla f\da\right) = 0, \ \ \ \ \ \ \ f(x,0) = \vf(x).
\end{equation}
Note that for this problem the assumptions \eqref{Ds} and \eqref{invmeas} are obviously satisfied, and in fact
a simple calculation shows that we presently have
\[
Q(t) = \frac{1- e^{-2t}}{2}\ I_m\ \nearrow\ Q_\infty = \frac 12 I_m. 
\]
One thus obtain from \eqref{dgamma}  the invariant measure
\[
d\gamma(x) = d\gamma_\infty(x) = (2\pi)^{-\frac m2} e^{-\frac{|x|^2}2} dx,
\]
i.e., the standard normalised Gaussian measure. 
The parabolic counterpart of the PDE in \eqref{cp0} is the well-known Ornstein-Uhlenbeck  equation
introduced in \cite{OU}, and further developed in \cite{WU}. Such equation  plays a significant role in probability, especially in connection with the calculus of Malliavin, see \cite[Chap. 2]{SS}. For a historical account up to the year 2018 the reader is also referred to the survey \cite{Bo2}.

Our first result concerning \eqref{cp0} is Proposition \ref{P:OUim}, which provides a representation formula for its solution. In order to state it, we introduce a basic functional class which is reminiscent of the Bargmann representation in a Fock space, see \cite[Chap. 1 \& 5]{Fo}. In what follows, given a function $\vf$, we indicate with $\psi$ the function
\begin{equation}\label{psi}
\psi(x) = e^{- \frac{|x|^2}{4}} \vf(x),\ \ \ \ \ \ x\in \Rm.
\end{equation} 
Notice from \eqref{psi} that 
\[
\psi\in L^2(\Rm)\ \Longleftrightarrow\ \vf\in L^2(\Rm,d\gamma),
\]
and that
\begin{equation}\label{phipsi}
||\psi||_{L^2(\Rm)} = ||\vf||_{L^2(\Rm,d\gamma)}.
\end{equation}
We denote by 
\begin{equation}\label{K}
\mathscr K(\Rm) = \{\vf\in C^\infty(\Rm)\mid \psi \in \mathscr S(\Rm)\}.
\end{equation}
It should be clear to the reader that $\mathscr K(\Rm)$ is dense in $L^2(\Rm,d\gamma)$. Throughout this paper we let $J^+ = (0,\pi)$, $J^- = (\pi,2\pi)$, and denote 
\[
J = J^+\cup J^-. 
\]

\begin{proposition}\label{P:OUim}
Let $\vf\in \mathscr K(\Rm)$. For every $x\in \Rm$ the function  
\begin{equation}\label{OUim}
f(x,t) = \begin{cases}
\frac{(4\pi
)^{-\frac{m}{2}} e^{\frac{imt}2}}{e^{\frac{i\pi m}{4}} (\sin t)^{\frac m2}} \int_{\Rm} e^{i \frac{|e^{it/2} y-e^{-it/2} x|^2}{4 \sin t}} \vf(y) dy,\ \ \ \ \ \ \ t\in J^+,
\\
\\
\frac{(4\pi
)^{-\frac{m}{2}} e^{\frac{imt}2}}{e^{\frac{i 3\pi m}{4}} |\sin t|^{\frac m2}} \int_{\Rm} e^{-i \frac{|e^{it/2} y-e^{-it/2} x|^2}{4 |\sin t|}} \vf(y) dy,\ \ \ \ \ \ \ t\in J^-,
\end{cases}
\end{equation} 
and $f(x,0) = \vf(x)$, solves \eqref{cp0} in $\Rm\times J$.
\end{proposition}

The next result extends to the Ornstein-Uhlenbeck group $e^{it\mathscr L}$ the well-known connection between the Fourier transform $\mathscr F$ and the free Schr\"odinger group $e^{it\Delta}$. We note that the definition of $\mathscr F$ adopted in this paper is \eqref{ft}, 
so that Plancherel theorem gives $||\mathscr F u||_{L^2(\Rm)} = ||u||_{L^2(\Rm)}$.
Henceforth, when $\vf\in L^2(\Rm,d\gamma)$, we will write $f(x,t) = e^{it\mathscr L}\vf(x)$.

\begin{proposition}\label{P:semigroup}
Let $\vf\in \mathscr K(\Rm)$. Then for every $t\in J^+$ the function $f(x,t) = e^{i t \mathscr L} \vf(x)$ is given by the formula
\begin{equation}\label{final}
e^{- \frac{|x|^2}{4}} f(x,t) =  (4\pi
)^{-\frac{m}{2}}\frac{e^{\frac{imt}2}}{e^{\frac{i\pi m}{4}} (\sin t)^{\frac m2}} e^{i \frac{\cot t |x|^2}{4}} \mathscr F\left(e^{i \frac{\cot t |\cdot|^2}4} \psi\right)(\frac{x}{4 \pi \sin t}),
\end{equation}
where $\psi$ is defined by \eqref{psi}.
\end{proposition} 

\begin{remark}\label{R:stable}
We explicitly note that it follows from \eqref{final} that
\begin{equation}\label{nice}
e^{i t \mathscr L} : \mathscr K(\Rm)\ \longrightarrow\ \mathscr K(\Rm).
\end{equation}
\end{remark}

To state the next proposition, we stress that for Schr\"odinger equations such as \eqref{cpHi}, or even the model problem \eqref{cp0}, restriction inequalities are at the moment unchartered territory. As a consequence, dispersive estimates such as \eqref{class} acquire an even stronger relevance in the implementation of the above mentioned approach of Ginibre-Velo to Strichartz estimates. In this connection, we have the following local result.
 
\begin{proposition}\label{P:disp}
Let $\vf\in \mathscr K(\Rm)$, and $f(x,t) = e^{it\mathscr L}\vf(x)$ with $t \not= k \pi$, with $k\in \mathbb Z$. For any $1\le p\le 2$ one has
\begin{equation}\label{disp}
||e^{-\frac{|\cdot|^2}4} f(\cdot,t)||_{L^{p'}(\Rm)} \le \left(\frac{p^{1/p}}{{p'}^{1/p'}}\right)^{\frac m2} \ \frac{1}{(4\pi |\sin t|)^{m(\frac 12 - \frac 1{p'})}}\ ||e^{-\frac{|\cdot|^2}4} \vf||_{L^p(\Rm)},
\end{equation}
where $1/p + 1/{p'} = 1$. The estimate \eqref{disp} is optimal, in the sense that it cannot possibly hold in the range $2<p\le \infty$. Moreover, equality is attained in \eqref{disp} by initial data $\vf(x) =  e^{-\alpha |x|^2+ \frac{|x|^2}4}$, $\Re \alpha>0$.  
\end{proposition}
Note that as $t\to 0$ the estimate \eqref{disp} displays the same behaviour $t^{-m(\frac 12 - \frac{1}{p'})}$ as that in \eqref{class} for the free Schr\"odinger group $e^{i t \Delta}$. 
It is interesting to compare \eqref{disp} with the following dispersive estimate for the Schr\"odinger equation with friction (see \eqref{Kol} above)
\begin{equation}\label{OUrd}
\p_t v - i \Delta v + \sa x,\nabla v\da = 0.
\end{equation}
The PDE \eqref{OUrd} is very different from \eqref{cp0}, as one can readily surmise from its invariance group of (real) left-translations 
\[
\sigma_{(x,s)} (y,t) = (x,s)\circ (y,t) = (y + e^{t}x,t+s).
\] 
As a special case of the result in \cite[Theorem 4.1]{GL}, one obtains for the semigroup $\{\mathcal T(t)\}_{t\ge 0}$ associated with the Cauchy problem for \eqref{OUrd}
\begin{equation}\label{feOU}
||\mathcal T(t) \vf||_{L^{p'}(\Rm)} \le C(m,p)\ \frac{e^{\frac{m}{p'} t}}{(1-e^{-2t})^{m(\frac 12 - \frac 1{p'})}} ||\vf||_{L^p(\Rm)}.
\end{equation}
The qualitative behaviour as $t\to 0^+$ in \eqref{disp} and \eqref{feOU} is the same, but because of the presence of $e^{-\frac{|\cdot|^2}4}$ in the former, the two estimates are incomparable.

Our final result is Proposition \ref{P:main} which represents a form of uniqueness for the group $e^{i t \mathscr L}$: a certain $L^2$ decay in Gaussian measure at two different times completely determines the solution. This kind of result is reminiscent of Hardy's uncertainty principle for the Fourier transform in his classical work \cite{Ha}. Nowadays, this subject is vast, and there exist many beautiful and important results scattered in the literature which is impossible to entirely quote here. The reader can see \cite{CP}, \cite{SST}, \cite{Fo} and \cite{Veluma} for an interesting account which covers up to 2004. Subsequently, the subject received a new impulse with the fundamental works \cite{EKPVcpde}, \cite{EKPVjems}, \cite{EKPVduke}, \cite{CEKPV}, \cite{EKPVjlms}, in which a whole new program of uncertainty inequalities for equations of Schr\"odinger type was developed. One should also see  the recent paper \cite{FM}. The reader familiar with the subject will surmise that the next proposition is one more manifestation of the ideas of Escauriaza, Kenig, Ponce and Vega in the above mentioned works. 

\begin{proposition}\label{P:main}
Assume that for some $a, b>0$ one has
\begin{equation}\label{L2}
||e^{a|\cdot|^2} f(\cdot,0)||_{L^2(\Rm, d\gamma)} + ||e^{b|\cdot|^2}  f(\cdot,s)||_{L^2(\Rm,d\gamma)}<\infty.
\end{equation}
If $a b \sin^2 s  \ge \frac{1}{16}$, then $f\equiv 0$ in $\Rm\times \R$. 
\end{proposition}

We note explicitly that the assumption $a b \sin^2 s\ge \frac{1}{16}$ automatically excludes the possibility that $s = k \pi$, with $k\in \mathbb Z$. This discrete set of points is where the complex covariance matrix $Q(t) = U(t) I_m$, where $U(t)$ is defined in \eqref{Qt0} below, becomes singular and the representation \eqref{OUim} of $e^{it\mathscr L}$ ceases to be valid. The proof of Proposition \ref{P:main} combines our formula \eqref{final} in Proposition \ref{P:semigroup} with the  $L^2$ version in \cite{CEKPV} of the above cited works \cite{EKPVcpde}-\cite{EKPVjlms}. We note that coupling Proposition \ref{P:semigroup} with the uncertainty principle for the Fourier transform due to Hardy, see \cite{Ha} or also \cite[Theorem 1.2]{CEKPV}, we obtain the following $L^\infty$ version of Proposition \ref{P:main}. 

\begin{proposition}\label{P:main2}
Suppose that for some $C, a, b>0$ one has for any $x\in \Rm$
\begin{equation}\label{Linfty}
|e^{-\frac{|x|^2}4}f(x,0)| \le C e^{-a|x|^2},\ \ \ \ \ \  |e^{-\frac{|x|^2}4} f(x,s)| \le  C e^{-b|x|^2}.
\end{equation}
If $a b \sin^2 s  \ge \frac{1}{16}$, then $f\equiv 0$ in $\Rm\times \R$. 
\end{proposition}

In closing we mention that there exists a deep link between \eqref{cp0} and the Cauchy problem for the quantum mechanics harmonic oscillator \eqref{H}. For the  problem \eqref{cpHi} such link becomes tenuous and it is all but clear (at least, to this author) what should replace the harmonic oscillator in such general framework. Since the present work, with the underlying Wiener space and invariant measure, serves as a model for the problem \eqref{cpHi}, we have not emphasised this link in the course of the paper, but have provided a brief (purely PDE) account in the appendix in Section \ref{S:app}. The reader interested in more theoretical aspects should consult e.g. \cite{Ta1}, \cite{Ta2}, \cite{BGV}, and especially the masterful lectures of Pauli \cite{Pauli}. Using the results in Section \ref{S:app}, one can derive from the dispersive estimate in Proposition \ref{P:disp} a corresponding result for the solution to \eqref{cp0}. In the opposite direction, Propositions \ref{P:main} and \ref{P:main2} are closely connected to results for the harmonic oscillator available in the literature, see the works \cite{Vel}, \cite{CF}, and the more recent \cite{KJO} and \cite{RR}, where the authors prove a conjecture by Vemuri, respectively up to an exceptional arithmetic progression, and in full. However, as far as we are aware of, in none of these papers the link with the Ornstein-Uhlenbeck operator and its invariant measure has been explicitly highlighted. 
 
The present work is organised as follows. In Section \ref{S:senzadrift} we establish the representation formula \eqref{OUim} in Proposition \ref{P:OUim} and we derive its main consequence,  Proposition \ref{P:semigroup}. With this result in hand, we then prove Proposition \ref{P:disp}. At the end of the section we prove Proposition \ref{P:main}.

\medskip

\textbf{Acknowledgement.} We thank B. Cassano, L. Fanelli, G. Folland, A. Lunardi, J. Ramos, S. Thangavelu and L. Vega for stimulating exchanges. 

\vskip 0.2in

%%%%%%%%%%%%%%%%%%%%%%%%%%%%%%%%%%%%%%%%%%%%%%
\section{The Schr\"odinger equation}\label{S:senzadrift}

In this section we solve the Cauchy problem \eqref{cp0} by constructing a representation of the solution operator. Our main result is the proof of \eqref{OUim} in Proposition \ref{P:OUim}, which we then use to prove \eqref{final} in Proposition \ref{P:semigroup} and \eqref{disp} in Proposition \ref{P:disp}.

\begin{proof}[Proof of Proposition \ref{P:OUim}]

We begin with a simple, but critical observation. Suppose that $v$ and $f$ are connected by the relation
\begin{equation}\label{drift0}
v(x,t) = f(e^{i t} x,t).
\end{equation}
Then, $f$ is a solution of the Cauchy problem \eqref{cp0} if and only if $v$ solves the problem
\begin{equation}\label{cpGsenzadrift00}
\begin{cases}
\p_t v - i U'(t) \Delta v = 0,
\\
v(x,0) = \vf(x),
\end{cases}
\end{equation} 
where we have let
\begin{equation}\label{Qt0}
U(t) =  \int_0^t e^{-2is} ds = \frac{1-e^{-2it}}{2i} = e^{-it} \frac{e^{it}-e^{-it}}{2i} = e^{-it} \sin t.
\end{equation}
To prove that $v$ solves \eqref{cpGsenzadrift00}, we argue as follows. The chain rule gives from \eqref{drift0}
\[
v_t(x,t) = i e^{it} \sa x,\nabla f(e^{it} x,t)\da + f_t(e^{it} x,t).
\]
On the other hand, the PDE in \eqref{cp0} gives
\[
f_t(e^{it} x,t) = i \Delta f(e^{it} x,t) - i e^{it} \sa x,\nabla f(e^{it} x,t)\da.
\]
Combining the latter two equations, we infer that $v$ solves 
\[
v_t(x,t) = i \Delta f(e^{it} x,t).
\]
Next, differentiating \eqref{drift0} we find
\[
\Delta v(x,t) = e^{2it} \Delta f(e^{it} x,t).
\]
We thus conclude that 
\[
v_t(x,t) = i e^{-2it} \Delta v(x,t) = i U'(t) \Delta v(x,t),
\]
where in the second equality we have used \eqref{Qt0}.
Summarising, the function $v$ solves the problem \eqref{cpGsenzadrift00}.
To find a representation formula for the latter, we use the Fourier transform.
Supposing that $v$ be a solution, we define
\begin{equation}\label{pFTv}
\hat v(\xi,t) = \mathscr F(v)(\xi,t) = \int_{\Rm} e^{-2\pi i\sa \xi,x\da} v(x,t) dx.
\end{equation}
Then \eqref{cpGsenzadrift00} is transformed into
\begin{equation}\label{cpGsenzadrift0}
\begin{cases}
\p_t \hat v + 4\pi^2 i U'(t)|\xi|^2 \hat v = 0,
\\
\hat v(\xi,0) = \hat \vf(\xi),
\end{cases}
\end{equation} 
whose solution is given by
\begin{equation}\label{hatv}
\hat v(\xi,t) = \hat \vf(\xi) e^{-4\pi^2 i U(t) |\xi|^2}.
\end{equation} 
Note that, with $U(t)$ as in \eqref{Qt0}, for every $t\in J$ the matrix $Q(t) = U(t) I_m$ is invertible. Moreover, we have
\begin{equation}\label{iQ}
i Q(t) = i U(t) I_m = i(\cos t - i \sin t) \sin t\ I_m = \sin^2 t\ I_m + i \frac{\sin{2t}}2 I_m. 
\end{equation}
We now invoke \cite[Theorem 7.6.1]{Hobook}, which we formulate as follows: Let $A\in G\ell(\mathbb C,m)$ be such that $A^\star = A$ and $\Re A \ge 0$. Then 
\begin{equation}\label{gengaussi2}
\mathscr F\left(\frac{(4\pi)^{-\frac{m}{2}}}{\sqrt{\operatorname{det} A}} e^{- \frac{\sa A^{-1}\cdot,\cdot\da}{4}}\right)(\xi) =
e^{- 4 \pi^2  \sa A\xi,\xi\da},
\end{equation}
where $\sqrt{\operatorname{det} A}$ is the unique analytic branch such that $\sqrt{\operatorname{det} A}>0$ when $A$ is real. 
If in \eqref{gengaussi2} we take 
\[
A = i Q(t) = i U(t) I_m =  i e^{-it} \sin t\ I_m,
\]
with  $t\in J^+$, then \eqref{iQ} gives $\Re A = \sin^2 t\ I_m \ge 0$, and
\[
A^{-1} = - i  \frac{e^{it}}{\sin t} I_m.
\]
We thus find 
\begin{equation}\label{larsetto2}
e^{-4\pi^2 i U(t) |\xi|^2}  = \mathscr F\left(\frac{e^{\frac{imt}2}}{e^{\frac{i\pi m}{4}} (\sin t)^{\frac m2}}(4\pi
)^{-\frac{m}{2}} e^{i e^{it}\frac{|\cdot|^2}{4 \sin t}}\right)(\xi).
\end{equation}
From \eqref{hatv} and \eqref{larsetto2} we conclude that for every $x\in \Rm$ and $t\in J^+$
\begin{equation}\label{hatv2}
v(x,t) = \frac{e^{\frac{imt}2}}{e^{\frac{i\pi m}{4}} (\sin t)^{\frac m2}}(4\pi
)^{-\frac{m}{2}} \int_{\Rm} e^{i e^{it}\frac{|y-x|^2}{4 \sin t}} \vf(y) dy.
\end{equation} 
Finally, keeping \eqref{drift0} in mind, after some elementary algebraic manipulations, we obtain the representation \eqref{OUim} when $t\in J^+$. The part corresponding to $t\in J^-$ follows by elementary changes if one observes that now $A = e^{i\frac{3\pi}2} e^{-it} |\sin t| I_m$. 

\end{proof}

It may be of interest to compare \eqref{OUim} with the well-known Mehler representation (see \cite{OU}, \cite{SS}, \cite{Bo2} and \cite{LMP}) 
\begin{align}\label{gustavo}
u(x,t) & = (4\pi)^{- \frac m2} e^{m t \sqrt \omega} \left(\frac{2\sqrt \omega}{\sinh(2t\sqrt \omega)}\right)^{\frac m2}
\\
& \times \int_{\Rm} \exp\left( -  \frac{\sqrt \omega}{2 \sinh(2t\sqrt \omega)} |e^{t\sqrt \omega} y - e^{-t\sqrt \omega} x|^2\right) \vf(y) dy
\notag
\end{align}
for the solution of the Cauchy problem for the Ornstein-Uhlenbeck operator
\begin{equation}\label{cpou}
\begin{cases}
u_t - \Delta u + 2 \sqrt \omega \langle x,\nabla u\rangle  = 0,\ \ \ \ \ \omega>0,
\\
u(x,0) = \vf(x).
\end{cases}
\end{equation}
 If one takes $\omega = \frac 14$, keeping in mind that $\sinh it = i \sin t$, then it is clear that by \emph{formally} substituting $t\to it$ in \eqref{gustavo}, one obtains the case $t\in J^+$ of \eqref{OUim}. Such formal manipulation is reminiscent of the physicist' Wick rotation, see \cite[Section 3]{Wi}.

\begin{proof}[Proof of Proposition \ref{P:semigroup}]
To further unravel \eqref{OUim}, and also to better clarify the role of the class $\mathscr K(\Rm)$ in \eqref{K}, note that if for $t\in J^+$ we expand
\begin{equation}\label{expa}
\frac{|e^{it/2}y-e^{-it/2} x|^2}{4 \sin t} = \frac{e^{it}|y|^2 + e^{-it}|x|^2 - 2\sa x,y\da}{4 \sin t},
\end{equation}
we find 
\begin{equation}\label{OUim4}
f(x,t) = \frac{(4\pi
)^{-\frac{m}{2}}e^{\frac{imt}2}}{e^{\frac{i\pi m}{4}} (\sin t)^{\frac m2}} \int_{\Rm} e^{i \frac{e^{it}|y|^2 + e^{-it}|x|^2 - 2\sa x,y\da}{4 \sin t}} \vf(y) dy.
\end{equation}
The change of variable $y = 4 \pi \sin t\ z$ in the integral in \eqref{OUim4} gives
\begin{align*}
& \int_{\Rm} e^{- i \frac{\sa x,y\da}{2 \sin t}} e^{i \frac{e^{it}|y|^2 + e^{-it}|x|^2}{4 \sin t}} \vf(y) dy = (4 \pi \sin t)^m e^{i \frac{e^{-it}|x|^2}{4 \sin t}} \int_{\Rm} e^{- 2 \pi i \sa x,z\da} e^{i \frac{e^{it}  |4 \pi \sin t\ z|^2}{4 \sin t}} \vf(4 \pi \sin t\ z) dz
\\
& = (4 \pi \sin t)^m e^{\frac{|x|^2}{4}} e^{i \frac{\cot t |x|^2}{4}} \int_{\Rm} e^{- 2 \pi i \sa x,z\da} e^{i \frac{\cot t |4 \pi \sin t\ z|^2}{4}} e^{- \frac{|4 \pi \sin t\ z|^2}{4}} \vf(4 \pi \sin t\ z) dz.
\end{align*}
Keeping \eqref{psi} in mind, 
we thus obtain from the above integral  
\begin{align*}
& \int_{\Rm} e^{- i \frac{\sa x,y\da}{2 \sin t}} e^{i \frac{e^{it}|y|^2 + e^{-it}|x|^2}{4 \sin t}} \vf(y) dy = (4 \pi \sin t)^m e^{\frac{|x|^2}{4}} e^{i \frac{\cot t |x|^2}{4}} \mathscr F\left(\delta_{4 \pi \sin t}\ e^{i \frac{\cot t |\cdot|^2}4} \psi\right)(x)
\\
& = e^{\frac{|x|^2}{4}} e^{i \frac{\cot t |x|^2}{4}} \mathscr F\left(e^{i \frac{\cot t |\cdot|^2}4} \psi\right)(\frac{x}{4 \pi \sin t}),
\end{align*}
where we have denoted by $\delta_\la f(x) = f(\la x)$ the action of the dilation operator on a function $f$.
Going back to \eqref{OUim4}, we have finally established \eqref{final}.

\end{proof}

Note that \eqref{final} shows that if $\vf\in \mathscr K(\Rm)$, then $f(\cdot,t)\in \mathscr K(\Rm)$,  and also
\[
||f(\cdot,t)||_{L^2(\Rm,d\gamma)} = ||\vf||_{L^2(\Rm,d\gamma)}.
\]

The equation \eqref{final} unveils the intertwining between the group $e^{i t \mathscr L}$ and the Fourier transform. 
We now use Proposition \ref{P:semigroup} to provide the 

\begin{proof}[Proof of Proposition \ref{P:disp}]
We first rewrite \eqref{final} in the following fashion
\begin{equation}\label{sfinalicchio}
\mathscr F\left(e^{i \frac{\cot t |\cdot|^2}4} \psi\right)(\frac{x}{4 \pi \sin t}) = (4\pi
)^{\frac{m}{2}}  \frac{e^{\frac{i\pi m}{4}}}{e^{\frac{imt}2}} (\sin t)^{\frac m2} e^{-i \frac{\cot t |x|^2}{4}} e^{- \frac{|x|^2}{4}} f(x,t).
\end{equation}
The identity \eqref{sfinalicchio} has the following direct consequence 
\begin{equation}\label{sfinalicchietto}
\left|\mathscr F\left(e^{i \frac{\cot t |\cdot|^2}4} \psi\right)(\frac{x}{4 \pi \sin t})\right| = (4\pi
)^{\frac{m}{2}} |\sin t|^{\frac m2} e^{- \frac{|x|^2}{4}} |f(x,t)|.
\end{equation}
If now $1\le p \le 2$, recall that in his celebrated paper \cite{Be} Beckner computed the sharp constant in the Hausdorff-Young inequality, and proved that for any $\psi\in L^p(\Rm)$ one has
\begin{equation}\label{HY}
\left(\int_{\Rm}|\mathscr F \psi(y)|^{p'} dy\right)^{\frac{1}{p'}} \le \left(\frac{p^{1/p}}{{p'}^{1/p'}}\right)^{\frac m2} \left(\int_{\Rm}|\psi(y)|^p dy\right)^{\frac 1p}.
\end{equation}
He also showed that equality is attained in \eqref{HY} if and only if $\psi$ is a Gaussian. The change of variable $y = \frac{x}{4 \pi \sin t}$ in \eqref{HY} gives
\begin{equation*}
\left(\int_{\Rm}|\mathscr F \psi(\frac{x}{4 \pi \sin t})|^{p'} dy\right)^{\frac{1}{p'}} \le \left(\frac{p^{1/p}}{{p'}^{1/p'}}\right)^{\frac m2} (4\pi |\sin t|)^{\frac{m}{p'}}\left(\int_{\Rm}|\psi(y)|^p dy\right)^{\frac 1p}.
\end{equation*}
We then turn to the proof of \eqref{disp}. Let $\vf\in \mathscr K(\Rm)$, and $f(x,t) = e^{it\mathscr L}\vf(x)$ with $t \in J^+$. We have from \eqref{sfinalicchietto} 
\begin{align*}
& \left(\int_{\Rm}|e^{- \frac{|x|^2}{4}} f(x,t)|^{p'} dy\right)^{\frac{1}{p'}} = (4\pi
)^{-\frac{m}{2}} |\sin t|^{-\frac m2} \left(\int_{\Rm} \left|\mathscr F\left(e^{i \frac{\cot t |\cdot|^2}4} \psi\right)(\frac{x}{4 \pi \sin t})\right|^{p'} dx\right)^{\frac{1}{p'}} 
\\
& \le (4\pi
)^{-\frac{m}{2}} |\sin t|^{-\frac m2} \left(\frac{p^{1/p}}{{p'}^{1/p'}}\right)^{\frac m2} (4\pi |\sin t|)^{\frac{m}{p'}}\left(\int_{\Rm} |\psi(x)|^{p} dx\right)^{\frac{1}{p}}. 
\end{align*}
Keeping in mind that $\psi(x) = e^{- \frac{|x|^2}{4}} \vf(x)$, see \eqref{psi}, we reach the desired conclusion \eqref{disp} from the latter inequality.

\end{proof}

We will  return to more general dispersive estimates for the group $e^{it\mathscr L}$ in a future study. 
 
%%%%%%%%%%%%%%%%%%%%%%%%%%%%%%%%%%%%%%%%%%%%%

%\section{Uncertainty again}\label{S:pf}

In the closing of this section we exploit Proposition \ref{P:semigroup} to prove Proposition \ref{P:main}.
We will need the following result, see \cite[Theorem 1.1]{CEKPV}. We note for the reader that our normalisation of the Fourier transform
\begin{equation}\label{ft}
\hat \vf(\xi) = \mathscr F(\vf)(\xi) = \int_{\Rm} e^{-2\pi i\sa \xi,x\da} \vf(x) dx,
\end{equation}
differs from theirs, and this accounts for the different constants in \eqref{hardyL2} below.

\begin{theorem}\label{T:har2}
Assume that $h:\Rm\to \R$ be a measurable function that satisfies 
\begin{equation}\label{hardyL2}
||e^{a|\cdot|^2} h||_{L^2(\Rm)} + ||e^{b|\cdot|^2} \hat h||_{L^2(\Rm)}<\infty.
\end{equation}
If $a b\ge \pi^2$, then $h\equiv 0$.  
\end{theorem} 

We are ready to give the 

\begin{proof}[Proof of Proposition \ref{P:main}]
Let $\vf\in L^2(\Rm,d\gamma)$ and $f(x,t) = e^{it\mathscr L}\vf(x)$. Let $a, b>0$ be as in the statement of the theorem, so that $\sin s\not=0$. With $\psi$ as in \eqref{psi}, consider the function defined by
\begin{equation}\label{ht}
h_s(x) = e^{i \frac{\cot s |x|^2}4} \psi(x).
\end{equation}
We have
\begin{align}\label{one}
\int_{\Rm} e^{2a |x|^2} |h_s(x)|^2 dx & = \int_{\Rm} e^{2a |x|^2} |\psi(x)|^2 dx = \int_{\Rm} e^{2a |x|^2} |\vf(x)|^2 e^{- \frac{|x|^2}{2}}dx
\\
& =  ||e^{a |x|^2} f(\cdot,0)||^2_{L^2(\Rm,d\gamma)} < \infty,
\notag
\end{align}
in view of \eqref{L2}. On the other hand, \eqref{final} in Proposition \ref{P:semigroup} gives  
\begin{equation*}
\left|\hat h_t\left(\frac{x}{4 \pi \sin t}\right)\right| = (4\pi
)^{\frac{m}{2}} (\sin t)^{\frac m2} e^{- \frac{|x|^2}{4}} |f(x,t)|,
\end{equation*}
and therefore for every $t\in J^+$ we have
\begin{align}\label{beauty}
& \left(\int_{\Rm} e^{2b |x|^2} \left|\hat h_t\left(\frac{x}{4 \pi \sin t}\right)\right|^2 dx\right)^{1/2} 
 = (4\pi
)^{\frac{m}{2}} (\sin t)^{\frac m2} \left(\int_{\Rm} e^{2b |x|^2} |f(x,t)|^2 e^{- \frac{|x|^2}{2}} dx\right)^{1/2}
\\
& = (4\pi
)^{\frac{m}{2}} (\sin t)^{\frac m2} ||e^{b |x|^2} f(\cdot,t)||_{L^2(\Rm,d\gamma)}.
\notag
\end{align}
If $s\in J^+$ is such that $||e^{b |x|^2} f(\cdot,s)||_{L^2(\Rm,d\gamma)} < \infty$, see \eqref{L2}, then we infer from \eqref{beauty} that 
\begin{equation}\label{nik}
\int_{\Rm} e^{2(16 b \pi^2 \sin^2 s) |y|^2} |\hat h_s(y)|^2 dy < \infty.
\end{equation}
In view of \eqref{one} and \eqref{nik}, applying Theorem \ref{T:har2} to the function $h_s$ we conclude that, if 
\[
16 \pi^2 a b  \sin^2 s \ge \pi^2\ \Longleftrightarrow\  a b \sin^2 s  \ge \frac{1}{16},
\]
then $h_s(x) = 0$ for every $x\in \Rm$. From \eqref{ht}, it is clear that this implies $\psi \equiv 0$, and therefore $\phi \equiv 0$, in $\Rm$.

\end{proof}

%%%%%%%%%%%%%%%%%%%%%%%%%%%%%%%%%%%%%%%%%%%%%%%%%%%%%%%%%%%%%%%%%%%

\section{Appendix: The imaginary harmonic oscillator}\label{S:app}

As said at the end of the introduction, for the benefit of the reader unfamiliar with the link between the Ornstein-Uhlenbeck operator and the quantum mechanics harmonic oscillator, in this brief appendix we show how to pass from one PDE to the other, and back. The main tool is the nonlinear Schr\"odinger equation of Riccati type \eqref{riccati} in Lemma \ref{L:ouho}. Consider the harmonic oscillator \footnote{This operator is usually defined as $H = \Delta - |x|^2$. We are using the $1/4$ normalisation in order not to have to change in $\Delta - 2 \sa v,\nabla\da$ that of the Ornstein-Uhlenbeck operator in the statement of Proposition \ref{P:connect} below.}  in $\Rm$
\begin{equation}\label{H}
H = \Delta - \frac{|x|^2}4
\end{equation}
and the Cauchy problem in $\Rm\times (0,\infty)$ for the associated Schr\"odinger operator
\begin{equation}\label{cpHO}
\p_t u - i H u = 0,\ \ \ \ \ \ \ \ u(x,0) = u_0(x),
\end{equation}
where the initial datum $u_0$ will be taken e.g. in $\mathscr S(\Rm)$. 
The following lemma establishes a general principle, one interesting consequence of which is that it allows to connect \eqref{cpHO} to the problem \eqref{cp0}, and in fact show that they are equivalent. The functions $\Phi$ and $h$ in its statement are assumed complex-valued.

\begin{lemma}\label{L:ouho}
Let $\Phi\in C(\R^{m+1})$ and $h\in C^2(\R^{m+1})$ be connected by the following  nonlinear Schr\"odinger equation
\begin{equation}\label{riccati}
i h_t + \Delta h - |\nabla h|^2  = \Phi.
\end{equation}
Then $u$ solves the partial differential equation
\begin{equation}\label{pdeho}
P u = i(\Delta u + \Phi u) - u_t = 0
\end{equation}
if and only if $f$ defined by the transformation
\begin{equation}\label{genexp}
u(x,t) = e^{-h(x,t)} f(x,t),
\end{equation} 
solves the equation
\begin{equation}\label{PDEou}
i(\Delta f - 2 \langle\nabla h,\nabla f\rangle) - f_t = 0.
\end{equation}
\end{lemma}

\begin{proof}
With $u$ as in \eqref{genexp}, we find
\begin{align*}
P u & = i \Delta(e^{-h} f) + i \Phi e^{-h} f - (e^{-h} f)_t
\\
& = i f \Delta(e^{-h}) + i e^{-h} \Delta f + 2 i \langle\nabla(e^{-h}),\nabla f\rangle + i \Phi e^{-h} f
 - (e^{-h})_t f -  e^{-h} f_t
\\
& = i e^{-h} f |\nabla h|^2 - i e^{-h} f \Delta h + i e^{-h} \Delta f - 2 i e^{-h} \langle\nabla h,\nabla f\rangle + i \Phi e^{-h} f
 - (e^{-h})_t f -  e^{-h} f_t
 \\
 & = e^{-h}\left\{i \left[\Phi - i h_t - (\Delta h - |\nabla h|^2 )\right] f + i \Delta f - 2 i \langle\nabla h,\nabla f\rangle - f_t\right\}.
\end{align*}
This computation proves that if $h$ and $\Phi$ solve \eqref{riccati}, then $u$ solves \eqref{pdeho} if and only if $f$ is a solution of \eqref{PDEou}.

\end{proof}

With Lemma \ref{L:ouho} in hand, we now return to \eqref{H} and prove the following result.
\begin{proposition}\label{P:connect}
A function $u$ solves the Cauchy problem \eqref{cpHO} if and only if the function
\begin{equation}\label{fu}
f(x,t) =  u(x,t)\ e^{\frac{|x|^2}4 + i \frac m2 t}
\end{equation}
solves \eqref{cp0} with $f(x,0) = \vf(x) = u_0(x)\ e^{\frac{|x|^2}4}$.
\end{proposition}

\begin{proof}
It is clear from \eqref{pdeho} that, in order to obtain from it the PDE in \eqref{cpHO}, we need  $\Phi(x,t) = - \frac{|x|^2}4$. With this choice, we look for a function $h(x,t)$ that is connected to such $\Phi$ by the equation \eqref{riccati}. A natural ansatz is $h(x,t) = A |x|^2 + B t$, with $A, B\in \mathbb C$ to be determined. Since $\Delta h = 2m A$, $h_t = B$ and $|\nabla h|^2 = 4 A^2 |x|^2$, to satisfy \eqref{riccati} we want
\[
i B + 2m A - 4 A^2 |x|^2 = - \frac{|x|^2}4,
\]
which holds iff $A = \frac{1}4, B =  i \frac m2$, and thus
\begin{equation}\label{h}
h(x,t) = \frac{|x|^2}4 + i \frac m2 t.
\end{equation}
With such choice of $h(x,t)$, the equation \eqref{genexp} in Lemma \ref{L:ouho} shows that $f(x,t)$ defined in \eqref{fu} solves the Cauchy problem \eqref{cp0}, with $f(x,0) = \vf(x) = u_0(x)\ e^{\frac{|x|^2}4}$. The ``if and only if" character of the statement is obvious.

\end{proof}

If $u_0\in \So$, then it is clear that $e^{\frac{|\cdot|^2}4} u_0\in \mathscr K(\Rm)$. According to Proposition \ref{P:connect}, we can express the group $e^{itH}$ by the formula
\begin{equation}\label{eitH}
e^{itH} u_0(x) = e^{-\frac{|x|^2}4 - i \frac m2 t} e^{it\mathscr L}(e^{\frac{|\cdot|^2}4} u_0)(x).
\end{equation}
Applying \eqref{OUim}, we infer from \eqref{eitH}.

\begin{corollary}\label{C:HOim}
Given $u_0\in \So$, for $t\in J^+$ one has
\begin{equation}\label{eitHg}
e^{itH} u_0(x) =  \frac{(4\pi
)^{-\frac{m}{2}} e^{-\frac{|x|^2}4}}{e^{\frac{i\pi m}{4}} (\sin t)^{\frac m2}} \int_{\Rm} e^{i \frac{|e^{it/2} y-e^{-it/2} x|^2}{4 \sin t}} e^{\frac{|y|^2}4} u_0(y) dy.
\end{equation}
\end{corollary}
Although with a different expression, formula \eqref{eitHg} is well-known to workers in harmonic analysis, see for instance \cite[Eq. (4.13), p.85]{Vel1}, or \cite[Eq. (6), p.164]{ST}.
Using \eqref{expa} in \eqref{eitHg}, we thus obtain the following counterpart of Proposition \ref{P:semigroup}.

\begin{corollary}\label{C:semigroup}
Given $u_0\in \mathscr S(\Rm)$, let $u(x,t) = e^{itH}u_0(x)$. Then for every $t\in J^+$ one has
\begin{equation}\label{finalino}
u(x,t) =  \frac{(4\pi
)^{-\frac{m}{2}}}{e^{\frac{i\pi m}{4}} (\sin t)^{\frac m2}} e^{i \frac{\cot t |x|^2}{4}} \mathscr F\left(e^{i \frac{\cot t |\cdot|^2}4} u_0\right)(\frac{x}{4 \pi \sin t}).
\end{equation}
\end{corollary} 

%%%%%%%%%%%%%%%%%%

\vskip 0.2in

\section{Declarations}

\noindent \textbf{Data availability statement:} This manuscript has no associated data.

\vskip 0.2in

\noindent \textbf{Funding and/or Conflicts of interests/Competing interests statement:} The author declares that he does not have any conflict of interest for this work

\bibliographystyle{amsplain}

\begin{thebibliography}{10}

\bibitem{BaBaGa}
H. Bahouri, D. Barilari \& I. Gallagher, \emph{Strichartz estimates and Fourier restriction theorems on the Heisenberg group}.
J. Fourier Anal. Appl. 27~(2021), no. 2, Paper No. 21, 41 pp.

\bibitem{BaFeGa}
H. Bahouri, C. Fermanian-Kammerer \& I. Gallagher, \emph{Dispersive estimates for the Schr\"odinger operator on step-$2$ stratified Lie groups}.
Anal. PDE 9~(2016), no. 3, 545-574.

\bibitem{BaGa}
H. Bahouri \& I. Gallagher, \emph{Local dispersive and Strichartz estimates for the Schr\"odinger operator on the Heisenberg group}.
Commun. Math. Res. 39~(2023), no. 1, 1-35.

\bibitem{BaPaXu}
H. Bahouri. P. G\'erard \& Chao-Jiang Xu, \emph{Espaces de Besov et estimations de Strichartz g\'en\'eralis\'ees sur le groupe de Heisenberg
Besov spaces and generalized Strichartz estimates on the Heisenberg group}.
J. Anal. Math. 82~(2000), 93-118.

\bibitem{Be}
W. Beckner, \emph{Inequalities in Fourier analysis}. Ann. of Math. (2) 102~(1975), no. 1, 159-182. 



\bibitem{Vel}
S. Ben Sa\"id, S. Thangavelu \& V. N. Dogga, \emph{
Uniqueness of solutions to Schr\"odinger equations on H-type groups}.
J. Aust. Math. Soc. 95~(2013), no. 3, 297-314.


\bibitem{BGV}
N. Berline, E. Getzler \& M. Vergne, \emph{Heat kernels and Dirac operators}. Grundlehren der Mathematischen Wissenschaften [Fundamental Principles of Mathematical Sciences], 298. Springer-Verlag, Berlin, 1992.

\bibitem{Bo}
V. I. Bogachev, \emph{Gaussian measures}.
Math. Surveys Monogr., 62
American Mathematical Society, Providence, RI, 1998, xii+433 pp.

\bibitem{Bo2} 
V. I. Bogachev, \emph{Ornstein-Uhlenbeck operators and semigroups}.
Uspekhi Mat. Nauk 73 (2018), no. 2, 3-74.
Russian Math. Surveys 73 (2018), no. 2, 191-260.

\bibitem{Bri}
H.C. Brinkman, \emph{Brownian motion in a field of force and the diffusion theory of chemical reactions. II}. Physica \textbf{23}~(1956), 149- 155.


\bibitem{BG&liar}
F. Buseghin, N. Garofalo \& G. Tralli, \emph{On the limiting behaviour of some nonlocal seminorms: a new phenomenon}.
Ann. Sc. Norm. Super. Pisa Cl. Sci. (5) 23 (2022), no. 2, 837-875.


\bibitem{CF}
B. Cassano \& L. Fanelli, \emph{Gaussian decay of harmonic oscillators and related models}.
J. Math. Anal. Appl. 456~(2017), no. 1, 214-228.


\bibitem{Caze}
 T. Cazenave, \emph{Semilinear Schr\"odinger equations}. Courant Lect. Notes Math., 10
New York University, Courant Institute of Mathematical Sciences, New York; American Mathematical Society, Providence, RI, 2003, xiv+323 pp.

\bibitem{Cha}
S. Chanillo, \emph{Uniqueness of solutions to Schr\"odinger equations on complex semi-simple Lie groups}.
Proc. Indian Acad. Sci. Math. Sci. 117 (2007), no. 3, 325-331.

\bibitem{CEKPV}
M. Cowling, L. Escauriaza, C. E. Kenig, G. Ponce \& L. Vega, \emph{The Hardy uncertainty principle revisited}.
Indiana Univ. Math. J. 59 (2010), no. 6, 2007-2025.

\bibitem{CP}
M. Cowling \& J.F. Price, \emph{Generalisations of Heisenberg's inequality}. Harmonic analysis (Cortona, 1982), 443-449.
Lecture Notes in Math., 992
Springer-Verlag, Berlin, 1983

\bibitem{DZ}
G. Da Prato and J. Zabczyk, \emph{Ergodicity for infinite-dimensional systems}, London Mathematical Society Lecture Note Series \textbf{229}~(1996), Cambridge University Press, Cambridge.

\bibitem{Hi}
M. Del Hierro, \emph{Dispersive and Strichartz estimates on $H$-type groups}.
Studia Math. 169~(2005), no. 1, 1-20.

\bibitem{EKPVcpde}
L. Escauriaza, C. E. Kenig, G. Ponce \& L. Vega, \emph{On uniqueness properties of solutions of Schr\"odinger equations}.
Comm. Partial Differential Equations 31 (2006), no. 10-12, 1811-1823.


\bibitem{EKPVjems}
L. Escauriaza, C. E. Kenig, G. Ponce \& L. Vega, \emph{Hardy's uncertainty principle, convexity and Schr\"odinger evolutions}.
J. Eur. Math. Soc. (JEMS) 10~(2008), no. 4, 883-907.

\bibitem{EKPVduke}
L. Escauriaza, C. E. Kenig, G. Ponce \& L. Vega, \emph{The sharp Hardy uncertainty principle for Schr\"odinger evolutions}.
Duke Math. J. 155~(2010), no. 1, 163-187.

\bibitem{EKPVjlms}
L. Escauriaza, C. E. Kenig, G. Ponce \& L. Vega, \emph{ Uncertainty principle of Morgan type and Schr\"odinger evolutions}.
J. Lond. Math. Soc. (2) 83 (2011), no. 1, 187-207.

\bibitem{FL}
B. Farkas \& A. Lunardi, \emph{Maximal regularity for Kolmogorov operators in $L^2$ spaces with respect to invariant measures}. J. Math. Pures Appl. 86~(2006), 310-321.

\bibitem{FM}
A. Fern\'andez-Bertolin \& E. Malinnikova, 
\emph{Dynamical versions of Hardy's uncertainty principle: a survey}. (English summary)
Bull. Amer. Math. Soc. (N.S.) 58~(2021), no.3, 357-375.


\bibitem{Fo}
G. B. Folland, \emph{Harmonic analysis in phase space}.
Ann. of Math. Stud., 122
Princeton University Press, Princeton, NJ, 1989, x+277 pp.

\bibitem{Fre}
M. Freidlin, \emph{Some remarks on the Smoluchowski-Kramers approximation}. J. Statist. Phys. \textbf{117}~(2004), no. 3-4, 617-634.

\bibitem{GL}
N. Garofalo \& A. Lunardi, \emph{Schr\"odinger semigroups and the H\"ormander hypoellipticity condition}. ArXiv: 5649756



\bibitem{GTmathann}
N. Garofalo \& G. Tralli, \emph{Hardy-Littlewood-Sobolev inequalities for a class of non-symmetric and non-doubling hypoelliptic semigroups}.
Math. Ann. 383~(2022), no. 1-2, 1-38.

\bibitem{GT}
N. Garofalo \& G. Tralli, \emph{Heat kernels for a class of hybrid evolution equations}.
Potential Anal. 59 (2023), no. 2, 823-856.

\bibitem{GV}
J. Ginibre \& G. Velo, \emph{On a class of nonlinear Schr\"odinger equations. I. The Cauchy problem, general case}.
J. Functional Analysis 32 (1979), no. 1, 1-32.

\bibitem{GVstrich}
J. Ginibre \& G. Velo, \emph{The global Cauchy problem for the nonlinear Schr\"odinger equation revisited}.
Ann. Inst. H. Poincar\'e Anal. Non Lin\'eaire 2~(1985), no. 4, 309-327.



\bibitem{Ha}
G. H. Hardy, \emph{A Theorem Concerning Fourier Transforms}.
J. London Math. Soc. 8 (1933), no. 3, 227-231.

\bibitem{Ho}
L. H{\"o}rmander,
\textit{Hypoelliptic second order differential equations}. Acta Math. 119~(1967), 147-171.

\bibitem{Hobook}
L. H{\"o}rmander,
\textit{The analysis of linear partial differential operators. I}.
Classics Math.
Springer-Verlag, Berlin, 2003, x+440 pp.

\bibitem{Kol}
A. N. Kolmogorov,  
\textit{Zuf\"allige Bewegungen (Zur Theorie der Brownschen Bewegung)}. Ann. of Math. (2) 35~(1934), 116--117.




\bibitem{KJO}
A. Kulikov, L. Oliveira \& J. P. G. Ramos, \emph{On Gaussian decay rates of harmonic oscillators and equivalences of related Fourier uncertainty principles}.
Rev. Mat. Iberoam. 40 (2024), no. 2, 481-502.

\bibitem{LP}
E. Lanconelli \& S. Polidoro, \emph{On a class of hypoelliptic evolution operators}.
Rend. Sem. Mat. Univ. Politec. Torino 52 (1994), no. 1, 29-63.

\bibitem{LM}
J. Ludwig \& D. M\"uller, \emph{Uniqueness of solutions to Schr\"odinger equations on $2$-step nilpotent Lie groups}.
Proc. Amer. Math. Soc. 142 (2014), no. 6, 2101-2118.

\bibitem{LMP}
A. Lunardi, G. Metafune \& D. Pallara, \emph{The Ornstein-Uhlenbeck semigroup in finite dimension}.
Philos. Trans. Roy. Soc. A 378 (2020), no. 2185, 20200217, 15 pp.

\bibitem{Mu}
D. M\"uller, 
\emph{A restriction theorem for the Heisenberg group}.
Ann. of Math. (2) 131~(1990), no.3, 567-587.

\bibitem{OU}
L. S. Ornstein \& G. E. Uhlenbeck, \emph{On the theory of the Brownian motion. I}.
Phys. Rev. (2) 36~(1930), 823-841. 

\bibitem{Pauli}
W. Pauli, \emph{Wave Mechanics}. Pauli Lectures on Physics, vol. 5. Dover 2000 edition of the original 1973 work published by The MIT Press, Cambridge, Massachusetts and London, England.  

\bibitem{RR}
D. Radchenko \& J. P. G. Ramos, \emph{Sharp Gaussian decay for the one-dimensional harmonic oscillator}, ArXiv:2305.18546

\bibitem{SS}
M. Sanz-Sol\'e, \emph{Malliavin calculus, with applications to stochastic partial differential equations}.
Fundam. Sci.
EPFL Press, Lausanne; distributed by CRC Press, Boca Raton, FL, 2005, viii+162 pp.

\bibitem{SST}
A. Sitaram, M. Sundari \& S. Thangavelu, \emph{Uncertainty principles on certain Lie groups}.
Proc. Indian Acad. Sci. Math. Sci. 105~(1995), no. 2, 135-151.

\bibitem{ST}
P. Sj\"ogren \& J. L. Torrea, \emph{On the boundary convergence of solutions to the Hermite-Schr\"odinger equation}.
Colloq. Math. 118~(2010), no. 1, 161-174.

\bibitem{Strex}
D. W. Stroock, \emph{An exercise in Malliavin's calculus}.
J. Math. Soc. Japan 67~(2015), no. 4, 1785-1799.

\bibitem{Str}
D. W. Stroock, \emph{Partial differential equations for probabilists}.
Cambridge Stud. Adv. Math., 112
Cambridge University Press, Cambridge, 2008. xvi+215 pp.



\bibitem{Ta1}
S. Takagi, \emph{Quantum dynamics and noninertial frames of reference. I. Generality}.
Progr. Theoret. Phys. 85~(1991), no. 3, 463-479.

\bibitem{Ta2}
S. Takagi, \emph{Quantum dynamics and noninertial frames of reference. II. Harmonic oscillators}.
Progr. Theoret. Phys. 85~(1991), no. 4, 723-742.

\bibitem{Vel1}
S. Thangavelu, \emph{Lectures on Hermite and Laguerre expansions}.
Math. Notes, 42
Princeton University Press, Princeton, NJ, 1993, xviii+195 pp.

\bibitem{Veluma}
S. Thangavelu, \emph{An introduction to the uncertainty principle}.
Progr. Math., 217
Birkh\"auser Boston, Inc., Boston, MA, 2004, xiv+174 pp.


\bibitem{WU}
M. C. Wang \& G. E. Uhlenbeck, \emph{
On the theory of the Brownian motion. II}.
Rev. Modern Phys. 17~(1945), 323-342.


\bibitem{Wi}
G. C. Wick, \emph{Properties of Bethe-Salpeter wave functions}.
Phys. Rev. (2) 96~(1954), 1124-1134.

\bibitem{Yo}
K. Yosida, \emph{Functional analysis}.
Classics Math.
Springer-Verlag, Berlin, 1995, xii+501 pp.

\bibitem{Z}
J. Zabczyk, \emph{Mathematical control theory: an introduction}.
Systems Control Found. Appl.
Birkh\"auser Boston, Inc., Boston, MA, 1992, x+260 pp.



\end{thebibliography}

\end{document}